\documentclass{amsart}
\usepackage{amssymb}
\usepackage{txfonts}
\usepackage{amsmath}
\usepackage{enumerate}
\usepackage{comment}
\usepackage{hyperref}

\newtheorem{thm}{Theorem}[section]

\newtheorem{lem}[thm]{Lemma}
\newtheorem{prop}[thm]{Proposition}

\theoremstyle{definition}
\newtheorem{defn}[thm]{Definition}
\theoremstyle{remark}
\newtheorem{rem}[thm]{Remark}
\numberwithin{equation}{section}

\def \R {\mathbb{R}}

\def \dist {\operatorname{dist}}

\begin{document}
	\title[]
	{pseudolocality and uniqueness of Ricci flow on almost Euclidean noncompact manifolds
	}
	
	\author{Liang Cheng,  Yongjia Zhang}

	

	
	\thanks{Liang Cheng's  research is partially supported by
		Natural Science Foundation of China 12171180. }
  \thanks{ Yongjia Zhang's research is partially supported by the starting-up grant of Shanghai Jiao Tong University.
	}
	
	\address{School of Mathematics and Statistics $\&$ Hubei Key Laboratory of Mathematical Sciences, Central  China Normal University, Wuhan, 430079, P.R.China}
	
	\email{chengliang@ccnu.edu.cn }
	
	\address{School of Mathematical Sciences, Shanghai Jiao Tong University, Shanghai, 200240, P.R.China	}
	
	\email{sunzhang91@sjtu.edu.cn}
	\maketitle

	\begin{abstract}
	In this paper, we prove a pseudolocality-type theorem for $\mathcal L$-complete noncompact  Ricci flow which may not have bounded sectional curvature; with the help of it we study the uniqueness of the Ricci flow on noncompact manifolds. In particular, we prove	
		the strong uniqueness theorem for the $\mathcal L$-complete Ricci flow
		 on the Euclidean space.  This partially answers a question proposed by  B-L.~Chen \cite{C}.
	\end{abstract}

	\section{Introduction}
	
	The Ricci flow is a geometric evolution equation introduced by Hamilton \cite{H1}, which deforms a Riemannian manifold by the Ricci curvature, namely,
	$$\frac{\partial}{\partial t}g_t=-2\mathrm{Ric}_{g_t}.$$
	On the one hand, the Ricci flow equation is nonlinear, which implies that almost certainly a Ricci flow develops singularity and does not exist for all time; see \cite{H2}. By a thorough analysis of the singularities in dimension three \cite{P1}, Perelman successfully overcame the obstacles in solving the Poincar\'e and geometrization conjuectures \cite{P1,P2,P3} using Hamilton's program \cite{H2}. On the other hand, the Ricci flow equaiton is parabolic, and consequently, many problems related to parabolic partial differential equations arise in this field. For instance, the short-time existence problem (c.f. \cite{De,H3,Sh}), the long-time existence problem (c.f. \cite{CW,H2,Se,W1,W2}), and the uniqueness problem (c.f. \cite{CL,K,LT}).

	Another interesting problem in the study of parabolic equations is that of the strong uniqueness. 
	In general, the solution to the initial value problem of a parabolic equation on an unbounded region is not unique.	
	For instance, the $1$-dimensional linear heat equation $u_t-u_{xx}=0$ on $\R^1$ admits a nontrivial solution with the initial value $u(\cdot,0)\equiv 0$, namely, the well-known Tychonoff example
	$$u(x,t)=\sum_{k=0}^\infty\frac{x^{2k}}{(2k)!}\frac{d^k}{dt^k}\exp\left(-\tfrac{1}{t^2}\right).$$
	So it is natural to ask whether or not the solution to the Ricci flow equation is always unique provided that the initial data is good. 
	
	In this respect, 	
	B-L.~Chen \cite{C}  proved a strong uniqueness result in dimension three, namely, starting from a $3$-dimensional complete Riemannian manifold with bounded and nonnegative sectional curvature, two Ricci flows must be identical to each other for a short time. As a consequence, a Ricci flow starting from the standard $3$-dimensional Euclidean space must always remain Euclidean. It is therefore interesting to ask whether such property is true for the Euclidean space in higher dimensions. We shall give an affirmative answer to this question with some additional assumptions.

	At the end of his paper \cite{C}, Chen mentioned that the strong uniqueness theorem can be obtained from a (more general) pseudolocality theorem, which states that the Ricci flow cannot evolve an almost Euclidean region immediately into a high-curvature region. Let us recall Perelman's well-known pseudolocality theorem \cite{P1} (c.f. \cite{CTY}).

	\begin{thm}[Perelman's pseudolocality theorem]\label{perelman_pseudo_locality}
		For any $\alpha>0$ there exist $\delta>0$ and $\epsilon_0>0$ depending on $\alpha$ and $n$ with the following preperty. Let  $(M^n,g_t)_{t\in [0,\epsilon^2]}$ be a complete Ricci flow on a noncompact manifold with bounded sectional curvature at each time on $(0,\epsilon^2)$, where $\epsilon \in (0,\epsilon_0]$, such that 
		\begin{equation}\label{small_scalar_cur}
		R_{g_0}\ge -1 \quad\text{ on }\quad B_{g_0}(x_0,1),
		\end{equation}
		and
		\begin{equation}\label{almost_euclidean_iso}
		\left(\operatorname{Area}_{g_0}(\partial \Omega)\right)^{n} \geq(1-\delta) n^n\omega_{n}\left(\operatorname{Vol}_{g_0}(\Omega)\right)^{n-1}
		\end{equation}
		for any regular domain $\Omega\subset B_{g_0}(x_0,1)$. 
		Then we have
		$$
		|Rm|(x,t)\le \frac{\alpha}{t}+\frac{1}{\epsilon^2},
		$$
		for any $d_{g_t}(x_0,x)<\epsilon$ and $t\in (0,\epsilon^2]$. 
	\end{thm}
	
The original version of Perelman's pseudolocality theorem \cite{P1} was proved
under the assumption of the manifold being closed. In the complete and noncompact case, this result was verified by Chau, Tam and Yu \cite{CTY}.  Later, Tian and Wang \cite{TW} proved another version of Perelman's
pseudolocality theorem in which they showed that the conditions (\ref{small_scalar_cur}) and (\ref{almost_euclidean_iso}) in Theorem \ref{perelman_pseudo_locality} can be replaced
by small Ricci curvature and almost Euclidean volume ratio.
 Wang \cite{W3} improved both Perelman's and Tian-Wang's pseudolocality theorems and proved that (\ref{small_scalar_cur}) and (\ref{almost_euclidean_iso}) in Theorem \ref{perelman_pseudo_locality} can be replaced
 by the smallness of the localized version of Perelman's $\mu$-functional.
One may also see \cite{CMZ} for another proof of the pseudolocality theorem based on Bamler's $\epsilon$-regularity theorem \cite{RB1}.

	It is to be noted that the above pseudolocality theorems  require bounded sectional curvature on the whole manifold within each compact  time-interval, under which condition the uniqueness of the Ricci flow is already established by Chen-Zhu \cite{CZ}.  Hence we cannot use the previous pseudolocality theorems to study the uniquess problems of the Ricci flow.
	In this paper we are bound to find a pseudolocality theorem valid in a more general setting, or at least with weaker regularity requirements after the initial time. It turns out that the $\mathcal{L}$-complete condition is a notion that fits our purpose.

	\begin{defn}\label{def:l-complete}
		A complete smooth Ricci flow $(M^n,g_t)_{t\in [0,T]}$ is called $\mathcal L$-complete on $M^n\times [0,T]$ if  the following holds. For any $s,\, t\in[0, T]$ with $s<t$ and any compact set $E\subset M$, there exists another compact set $F\subset M$ that contains $E$, such that if $x,\, y\in E$, then there exists at least one minimal $\mathcal{L}$-geodesic connecting $(x,t)$ and $(y,s)$, and \emph{all} possible minimal $\mathcal{L}$-geodesics connecting $(x,t)$ and $(y,s)$ are contained (spatially) in $F$.
	\end{defn}
	
	First of all, we will show that under certain conditions the Ricci flow is $\mathcal L$-complete. 	
	
	\begin{thm}\label{conditions_l-complete}
		Let  $(M^n,g_t)_{t\in [0,T]}$ be a Ricci flow such that \textbf{either} one of the following is true.
		\begin{enumerate}[(1)]
			\item $R_{g_t}(x)\geq -k\quad\text{for all}\quad (x,t)\in M\times[0,T]$, where $k$ is a nonnegative constant, and there exist a positive constant $c$ and a complete smooth Riemannian metric $\bar g$ on $M$, such that 
			\begin{align*}
				g_t\ge c\bar g\quad\text{ for all }\quad t\in[0,T].
			\end{align*}
			\item $R_{g_t}(x)\geq -k\quad\text{for all}\quad (x,t)\in M\times[0,T]$, where $k$ is a nonnegative constant,  and there exist constants $\alpha>0$, $\beta>0$, $\gamma\in(0,1)$, and a fixed point $x_0\in M$, such that
			\begin{align*}
				\mathrm{Ric}_{g_t}(x)\geq -\alpha t^{-\gamma}\ln\left(\dist_{g_t}(x_0,x)+\beta\right)g_t\quad\text{ for all }\quad (x,t)\in M\times [0,T].
			\end{align*}
			Moreover, $T\le \overline{T}(\alpha,\gamma)$, where $\overline{T}(\alpha,\gamma)$ signifies a positive constant depending only on $\alpha$ and $\gamma$.
		\end{enumerate}
		Then $(M^n,g_t)_{t\in[0,T]}$ is $\mathcal{L}$-complete on $[0,T]$.
	\end{thm}

	On the other hand, the significance of the $\mathcal{L}$-complete assumption is that with it a version of Perelman's pseudolocality theorem can be proved. Indeed, it is well-known that Perelman's entropy and reduced geometry reflect the same geometric property of the Ricci flow from differently perspectives. While the bounded curvature condition enables the estimates for the entropy, through which the pseudolocality theorem was proved, the $\mathcal{L}$-complete assumption leads to a similar conclusion via the reduced geometry. In this way, we prove the following pseudolocality-type theorem for $\mathcal{L}$-complete Ricci flows.

	\begin{thm}\label{pseudo_locality}
		For any $\alpha>0$  there exist $\delta>0$ and $\epsilon_0>0$ depending on $\alpha$, and $n$ with the following property. Let  $(M^n,g_t)_{t\in [0,\epsilon^2]}$ be a smooth $\mathcal L$-complete Ricci flow on a noncompact manifold, where $\epsilon \in (0,\epsilon_0]$, such that 
		$$
		R(g_0)\ge -1,
		$$
		and
		$$
		\left(\operatorname{Area}_{g_0}(\partial \Omega)\right)^{n} \geq(1-\delta) n^n\omega_{n}\left(\operatorname{Vol}_{g_0}(\Omega)\right)^{n-1}
		$$
		for any regular domain $\Omega\subset M$. 
		Then we have
		$$
		|Rm|(x,t)\le \frac{\alpha}{t}\qquad\text{for all}\qquad (x,t)\in M\times(0,\epsilon^2].
		$$
		
	\end{thm}
	
	Applying the above pseudolocality theorem, we can study the strong uniqueness of the $\mathcal L$-complete Ricci flow in higher dimensions. In particular,
 we show the strong uniqueness theorem for the $\mathcal L$-complete Ricci flow on Euclidean space in higher dimensions,  which partially answers a question proposed by  B-L.~Chen \cite{C}.
	
	\begin{thm}\label{uniqueness}
		Let  $(M^n,g_t)_{t\in[0,T]}$ be a smooth $\mathcal L$-complete Ricci flow on a noncompact manifold $M$ such that 
		$$
		|\mathrm{Rm}_{g_0}|\le C,
		$$
		and
		$$
		\left(\operatorname{Area}_{g_0}(\partial \Omega)\right)^{n} \geq(1-\delta) n^n\omega_{n}\left(\operatorname{Vol}_{g_0}(\Omega)\right)^{n-1}
		$$
		for any regular domain $\Omega\subset M$. 
		Then we have
		$$
		|\mathrm{Rm}_{g_t}|\le C\qquad\text{for all}\qquad t\in[0,T].
		$$
		Consequently, the solution $g_t$, $t\in[0,T]$, is uniquely determined by $g_0$.	In particular, two $\mathcal L$-complete Ricci flows starting from a Riemannian manifold satisfying the assumptions of the theorem must always be isometric to each other.
	\end{thm}

	Combining Theorem \ref{conditions_l-complete}(1)(2) and Theorem \ref{uniqueness},  we obtain the following result.

	\begin{thm}\label{main0}
		Let $(\mathbb{R}^n,g_t)_{t \in[0, T]}$ be a complete and smooth solution to the Ricci flow with $g_0=g_E$, where $g_E$ is the standard Euclidean metric. Assume that
		\textbf{either} one of the following is true.
		\begin{enumerate}[(1)]
			\item there exists a positive constant $c$, such that 
			\begin{align*}
				g_t\ge c\bar g\quad\text{ for all }\quad t\in[0,T],
			\end{align*}
			where $\bar g$ is a complete smooth Riemannian metric on $\mathbb{R}^n$
			\item  there exist constants $\alpha>0$, $\beta>0$, $\gamma\in(0,1)$, and a fixed point $x_0\in\mathbb{R}^n$, such that
			\begin{align*}
				\mathrm{Ric}_{g_t}(x)\geq -\alpha t^{-\gamma}\cdot\ln\left(\dist_{g_t}(x_0,x)+\beta\right)g_t\quad\text{ for all }\quad (x,t)\in \mathbb{R}^n\times [0,T].
			\end{align*}	
		\end{enumerate}		
		Then we have $g_t\equiv g_E$ for all $t\in[0,T]$.
	\end{thm}

	\textbf{Acknowledgement.} The  authors would like to thank Professor Jiaping Wang for some helpful suggestions.

	\section{Preliminaries}
	
	In this section, we collect some well-known results applied in the proofs of our main theorems. These results are chiefly about Perelman's reduced distance and reduced volume; most of them can be found in \cite{P1,Y1}. 
	Let $(M^n,g(\tau))_{\tau\in[0,T]}$ be a solution to the backward Ricci flow equation
	\begin{align*}
		\frac{\partial }{\partial \tau}g(\tau)=2\mathrm{Ric}_{g(\tau)},
	\end{align*}
	where $\tau$ stands for the backward time. In fact, if $(M^n,g_t)_{t\in[0,T]}$ is a Ricci flow, then $g(\tau):=g_{T-\tau}$ solves the backward Ricci flow equation. In the remaining part of the paper, since we often shift between Ricci flow and backward Ricci flow, we remind the reader that if the ``time variable'' is a subindex, such as $g_t$, then the notation stands for a Ricci flow; if the ``time variable'' is placed in a pair of parentheses, such as $g(\tau)$, then the notation represents a backward Ricci flow.
	
	Let $\gamma(s):[0,\tau]\to M$, where $\tau\in(0, T]$, be a piecewise $C^1$ curve, then Perelman's $\mathcal L$-energy is defined as
	\begin{align}\label{l-energy} \mathcal{L}(\gamma)=\int^{\tau}_0
		\sqrt{s}\left(R_{g(s)}(\gamma(s))+|\gamma'(s)|_{g(s)}^2\right)ds.
	\end{align}	
	One may view $\gamma$ as a curve in the (backward) Ricci flow space-time from $(\gamma(0),0)$ to $(\gamma(\tau),\tau)\in M\times[0,T]$, satisfying $(\gamma(s),s)\in M\times\{s\}$ for all $s\in[0,\tau]$. 
	Let $x_0\in M$ be a fixed base point. For any $(x,\tau)\in M\times(0,T]$, we define
	\begin{align}\label{reduced_l} 
		L(x,\tau)=\inf\limits_{\gamma} \mathcal{L}(\gamma(s)),\quad \ell(x,\tau)=\frac{L(x,\tau)}{2\sqrt{\tau}},
	\end{align}
	where the infimum is taken over all  piecewise $C^1$
	curves $\gamma(s):[0,\tau]\to M$ satisfying $\gamma(0)=x_0$ and
	$\gamma(\tau)=x$. A minimizer of $L$ (should it exist) is called a \emph{minimal $\mathcal{L}$-geodesic}, and $\ell(\cdot,\cdot): M\times(0,T]\to \R$ is called Perelman's reduced distance function or the $\ell$-function based at $(x_0,0)$.  
	
	\begin{rem}
		When considering a forward Ricci flow $(M,g_t)_{t\in[0,T]}$, then a minimal $\mathcal{L}$-geodesic is said to connect $(x,t)$ and $(y,s)$, where $s<t$, if it minimizes \eqref{l-energy} with respect to the backward Ricci flow $(M,g(\tau):=g_{t-\tau})_{\tau\in[0,t]}$, starts at $(x,0)$, and ends at $(y,t-s)$.
	\end{rem}
	
	For $v\in T_{x_0}M,$ let $\gamma_v$ denote the $\mathcal{L}$-geodesic satisfying $\lim\limits_{s\to 0}\sqrt{s}\gamma'(s) = v$. If $\gamma_v$ exists on $[0,\tau]$, then the $\mathcal{L}$-exponential map $\mathcal{L}\text{exp}^{\tau}_{x_0}:T_{x_0}M\to M$ at time $\tau$ is defined as
	$$\mathcal{L}\text{exp}^{\tau}_{x_0}(v)=\gamma_v(\tau).$$
	Let $U(\tau)\subset T_{x_0}M\cong\mathbb{R}^n$ be the maximal domain of $\mathcal{L}\text{exp}^{\tau}_{x_0}$.  By applying the basic ODE theory to the $\mathcal{L}$-geodesic equation---a linear ODE depending only on the geometry near the $\mathcal{L}$-geodesic---one may conclude that $U(\tau)$ is an open set and that $\mathcal{L}\text{exp}^{\tau}_{x_0}$ is a smooth map from $U(\tau)$ to $M$. The injectivity domain at time $\tau$ is defined as
	\begin{align*}
		&&\Omega^{T_{x_0}M}(\tau)=\Big\{ v\in U(\tau)\ \Big\vert \ \gamma_v|_{[0,\tau]} :[0,\tau] \to M \text{ is the unique minimal } \mathcal{L}\text{-geodesic}\\
		&&
		\text{\ from $(x_0,0)$ to $(\gamma_v(\tau),\tau)$ }; \text{$\gamma_v(\tau) $ is not conjugate to }x_0 \text{ along\ } \gamma_v.\Big\}.
	\end{align*}
	It is easily seen that
	\begin{align}\label{Injective domain}
		\Omega^{T_{x_0}M}(\tau_2)\subset \Omega^{T_{x_0}M}(\tau_1)\qquad \text{ if }\quad \tau_1\le\tau_2.
	\end{align}
	Correspondingly we also define
	\begin{align*}
		&&\Omega(\tau)=\Big\{q\in M\ \Big\vert\  \text{There is a unique minimal } \mathcal{L}-\text{geodesic }
		\gamma:[0,\tau] \to M
		\text{\ with\ }\\&& \gamma(0) = x_0, \gamma(\tau) = q ; q\text{ is not conjugate to } x_0\text{ along\ } \gamma.\Big\}.
	\end{align*}
	Obvisously,
	$$
	\Omega(\tau)=\mathcal{L}\text{exp}^{\tau}_{x_0}\left(\Omega^{T_{x_0}M}(\tau)\right),
	$$
	and Perelman's reduced volume based at $(x_0,0)$ is defined as
	\begin{align}\label{RV}
		\mathcal {V}(\tau)=\int_{\Omega(\tau)} (4\pi \tau)^{-\frac{n}{2}}
		e^{-\ell(\cdot,\tau)}
		d\mu(\tau).
	\end{align}
	The $\mathcal{L}$-cut-locus
	is defined as
	$$
	C(\tau)=M\backslash \Omega(\tau).
	$$
	
	Let
	$J_i^V(\tau), i=1,\cdots,n,$ be $\mathcal{L}$-Jacobi fields along
	$\gamma_V(\tau)$ with $J_i^V(0)=0, (\nabla_V J_i^V)(0)=E^0_i,$ where
	$\{E^0_i\}^n_{i=1}$ is an orthonormal basis on $T_{x_0}M$ with respect
	to $g(0)$. Then $D(\mathcal{L} exp^\tau)\big|_V(E^0_i)=J^V_i(\tau)$,
	and the reduced volume can also be computed as
	\begin{align}\label{WFRV}
		\mathcal {V}(\tau)=\int_{	\Omega^{T_{x_0}M}(\tau)} (4\pi \tau)^{-\frac{n}{2}}
		e^{-l(\gamma_V(\tau),\tau)}
		\mathcal{L}J_V(\tau)dx_{g(0)}(V),
	\end{align}
	where
	$\mathcal{L}J_V(\tau)=\sqrt{det\left(\left<J^V_i(\tau),J^V_j(\tau)\right>\right)}$
	and $dx_{g(0)}$ is the standard Euclidean volume form on
	$(T_{x_0}M,g(x_0,0))$.

	To end this section, we recall the famous logarithmic Sobolev inequality due to Gross \cite{G}. This inequality is important in the proofs of our main theorems, since it is strongly related to Perelman's monotonicity formulas. The particular form of the following theorem can be found in \cite[Theorem 22.5]{RFV3}.
	
	\begin{thm}[Logarithmic Sobolev inequality]\label{logsob}
		Let $(M^n,g)$ be a complete Riemannian manifold which satisfies
		$$
		\left(\operatorname{Area}_{g}(\partial \Omega)\right)^{n} \geq I_n\left(\operatorname{Vol}_{g}(\Omega)\right)^{n-1}
		$$
		for any regular compact domain $\Omega\subset M^n$.
		Then for any $W^{1,2}$ function $\varphi$ on $M^n$, we have
		\begin{align*}
			\int_{M^n}\left(2|\nabla\varphi|^2-\varphi^2\log\varphi^2\right)d\mu+\log\left(\int_{M^n}\varphi^2 d\mu\right)\int_{M^n}\varphi^2 d\mu\geq \left(\frac{n}{2}\log(2\pi)+n+\log\left(\frac{I_n}{c_n}\right)\right)\int_{M^n}\varphi^2 d\mu,
		\end{align*}
		where $c_n=n^n\omega_n$ is
		the isoperimetric constant of Euclidean space.
	\end{thm}

	\section{Perelman's $\mathcal{L}$-geometry on $\mathcal L$-complete Ricci flow}
	
	In this section,  we shall first show that Perelman's theory of $\mathcal L$-geometry can be extended to the case of $\mathcal L$-complete Ricci flows, and then show that the Ricci flow is $\mathcal L$-complete under
	the assumptions made in Theorem \ref{conditions_l-complete}.

	\subsection{$\mathcal L$-complete Ricci flow}
	In Perelman's study \cite{P1} of the $\mathcal L$-geometry, a general (although implicit) assumption is bounded sectional curvature. However, Ye \cite{Y1} studied the properties of the $\ell$-function and the reduced volume assuming only a lower bound of the Ricci curvature. 	With a little more observation,  Perelman's theory of $\mathcal L$-geometry can be extended to $\mathcal L$-complete Ricci flows. 
	We will leave it to the reader to check that most proofs in \cite{Y1} are valid, and we summarize some important results below.
	
	\begin{thm}[\cite{P1}, see also {\cite[Proposition 2.14, Lemma 2.22, Theorem 2.23]{Y1}}]\label{Perelman}
		Let $(M^n,g(\tau))_{\tau\in[0,T]}$ be a smooth $\mathcal L$-complete backward Ricci flow. Let $ \ell$ be the reduced distance function based at $(x_0,0)$. Then $\ell$ is locally Lipschitz on $M\times(0,T]$ and the following equation and inequalities hold almost everywhere in the smooth sense on $ M\times(0,T]$.
		\begin{align}
			&2\frac{\partial \ell}{\partial \tau}+|\nabla l|^2-R+\frac{\ell}{\tau}=0,\label{eq_l_1}\\
			&	\frac{\partial }{\partial \tau}\ell-\Delta \ell +	|\nabla \ell|^2-R+\frac{n}{2\tau} \geq 0,\label{eq_l_4}\\
			&	2\Delta \ell-|\nabla \ell|^2+R+\frac{\ell-n}{\tau} \le 0.\label{eq_l_5}
		\end{align}
		Furthermore, (\ref{eq_l_4}) and (\ref{eq_l_5}) both hold on $M\times (0,T]$ in the sense of distribution. That is to say, for any $0<\tau_1<\tau_2\le T$ and for any nonnegative Lipshcitz function $\phi$ compactly supported on $M\times [\tau_1,\tau_2]$, it holds that
		\begin{eqnarray}\label{eq_l_6}
			\int_{\tau_1}^{\tau_2}\int_M\left(\nabla \ell\cdot\nabla\phi+\left(\frac{\partial}{\partial\tau}\ell+|\nabla \ell|^2-R+\frac{n}{2\tau}\right)\phi\right)dg(\tau)d\tau\geq0,
		\end{eqnarray}
		and, for any $\tau\in(0,T]$ and any nonnegative Lipshcitz function $\phi$ compactly supported on $M$, it holds that
		\begin{eqnarray}\label{eq_l_7}
			\int_M\left(-2\nabla \ell\cdot\nabla\phi+\left(-|\nabla \ell|^2+R+\frac{\ell-n}{\tau}\right)\phi\right)dg(\tau)\leq 0,
		\end{eqnarray}
		Moreover, the integrand in (\ref{WFRV}) is pointwise monotonically non-increasing, namely,
		\begin{eqnarray}\label{eq_l_8}
			\frac{d}{d \tau}\left((4\pi \tau)^{-\frac{n}{2}}
			e^{-l(\gamma_V(\tau),\tau)}
			\mathcal{L}J_V(\tau)\right)\leq 0,
		\end{eqnarray}
		for any $V\in 	\Omega^{T_{x_0}M}(\tau)$.
	\end{thm}
	\begin{proof}
		By virtue of our definition of the $\mathcal{L}$-complete Ricci flow (Definition \ref{def:l-complete}), when analyzing the local properties of the $\mathcal{L}$-geodesics or the reduced distance, only local geometry is involved. Thus the proofs in \cite{Y1} can be wholly adopted in our case.
		
		We emphasize that although according to Ye's formulation, \eqref{eq_l_7} and \eqref{eq_l_8} hold only on $M\times (0,T)$, yet this restriction stems from the application of Shi's estimates when obtaining the local curvature derivative bounds. However, in our assumption, the backward Ricci flow is assumed to be smooth on $M\times[0,T]$, namely, each curvature derivative is bounded locally. Thus  \eqref{eq_l_7} and \eqref{eq_l_8} are also valid on $M\times(0,T]$.
	\end{proof}

	\begin{thm}[\cite{P1}, see also {\cite[Theorem 4.3, Theorem 4.5]{Y1}}]\label{Monotonicity}
		Let $(M^n,g(\tau))_{\tau\in[0,T]}$ be a $\mathcal L$-complete backward Ricci flow.
		Let $\mathcal {V}(\tau)$ be the reduced volume based at $(x_0,0)$. Under the same assumptions as in the above theorem, we have
		\begin{enumerate}
			\item $0<\mathcal {V}(\tau)\leq 1$ for all $\tau\in(0,T]$;
			\item $\mathcal {V}(\tau)$ is non-increasing in $\tau$;
			\item if $\mathcal V(\tau)=1$ for some $\tau\in (0,T]$, then $M^n=\mathbb{R}^n$ and $g(s)\equiv g_E$ for all $s\in[0,\tau]$, where $g_E$ is the standard Euclidean metric.
		\end{enumerate}
	\end{thm}	
	\begin{proof}
		According to the definition of the reduced volume \eqref{WFRV}, this theorem is a consequence of \eqref{Injective domain} and \eqref{eq_l_8}. The details can be found in \cite{Y1}.
	\end{proof}

	\subsection{Sufficient conditions of $\mathcal{L}$-completeness}
	In this subsection, we shall prove Theorem \ref{conditions_l-complete}. To begin with, we prove the following straightforward auxilliary lemma.
	
	\begin{lem}[bounded $\mathcal{L}$-energy in compact set]\label{L-bounded in compact sets}
		Let $(M,g(\tau))_{\tau\in[0,T]}$ be a complete backward Ricci flow. For any $\bar\tau\in(0,T]$ and any compact set $E\in M$, there exists a positive number $L$, such that for any $x_0$, $x\in E$, there is a picewise $C^1$ curve $\gamma:[0,\bar\tau]$ with $\gamma(0)=x_0$ and $\gamma(\bar\tau)=x$, satisfying 
		$$\mathcal{L}(\gamma)\le L.$$
	\end{lem}
	
	\begin{proof}
		Since $E$ is compact and the backward Ricci flow is complete, we can find a $y_0\in M$ and fix $A>0$ and $K>0$, such that
		\begin{align*}
			E\subset B_{g(0)}(y_0,A)\qquad \text{and}\qquad \left|\mathrm{Ric}\right|\le K\ \text{ on }\ B_{g(0)}(y_0,3A)\times [0,\bar\tau].
		\end{align*}
		For any $x_0$, $x\in E$, let $\gamma:[0,\bar\tau]\to M$ a normalized shortest $g(0)$-geodesic with $\gamma(0)=x_0$ and $\gamma(\bar\tau)=x$. Then it is obvious that
		\begin{align*}
			\gamma(\tau)\in B_{g(0)}(y_0,3A)\ \ \text{ and }\  \ |\gamma'(\tau)|_{g(0)}=\frac{\dist_{g(0)}(x_0,x)}{\bar\tau}\le\frac{2A}{\bar\tau},\ \ \text{ for all }\ \tau\in[0,\bar\tau].
		\end{align*}
		Thus, we compute
		\begin{align*}
			\mathcal{L}(\gamma)&=\int_{0}^{\bar\tau}\sqrt{\tau}\left(R_{g(\tau)}(\gamma(\tau))+\left|\gamma'(\tau)\right|^2_{g(\tau)}\right)\, d\tau
			\\
			&\le \frac{2}{3}\sqrt{n}K\bar\tau^{\frac{3}{2}}+e^{K\bar\tau}\int_0^{\bar\tau}\sqrt{\tau}\left|\gamma'(\tau)\right|^2_{g(0)}\, d\tau
			\\
			&\le\frac{2}{3}\sqrt{n}K\bar\tau^{\frac{3}{2}}+\frac{8A^2e^{K\bar\tau}}{3\sqrt{\bar\tau}}.
		\end{align*}
		Letting $L$ be the constant on the last line above finishes the proof of the lemma.
	\end{proof}

	\subsubsection{Lower uniform equivalence} We prove Theorem \ref{conditions_l-complete}(1). Consider a backward Ricci flow $(M,g(\tau))_{\tau\in[0,T]}$ satisfying
	\begin{align}\label{Rlowerbound}
		R_{g(\tau)}(x)\geq -k\quad&\text{for all}\quad (x,\tau)\in M\times[0,T],
		\\
		g(\tau)\geq c\bar g\quad&\text{for all}\quad \tau\in[0,T].\label{equivalence}
	\end{align}
	where $k\ge 0$, $c>0$, and $\bar g$ is a complete smooth Riemannian metric $\bar g$ on $M$.
	The central idea of proving the $\mathcal{L}$-completeness is to show that any space-time curve with bounded $\mathcal{L}$-energy lies in a space-time compact set; this idea is also used in the following part of the subsection.

	\begin{prop}\label{existence}
		Under the assumptions (\ref{Rlowerbound}) and (\ref{equivalence}), the following holds. For any $(x_0,0)$ and $(x,\tau)\in M\times(0,T]$, there exists a minimal $\mathcal L$-geodesic $\gamma:[0,\tau]\to M$ satisfying $\gamma(0)=x_0$ and $\gamma(\tau)=x$. Furthermore, the backward Ricci flow is $\mathcal{L}$-complete on $M\times [0,T]$.
	\end{prop}
	\begin{proof}		
		Let us consider any piecewise $C^1$ curve $\gamma:[0,\tau]\to M$ satisfying $\gamma(0)=x_0$ and $\mathcal L(\gamma)\leq L$. Then, by (\ref{Rlowerbound}) and (\ref{equivalence}), we may estimate
		\begin{align*}
			L \ge \mathcal L(\gamma)&=\int_0^\tau\sqrt{s}\left(R_{g(s)}(\gamma(s))+|\gamma'(s)|^2_{g(s)}\right)ds
			\\
			&\ge -\frac{2k}{3}\tau^{\frac{3}{2}}+\int_0^\tau \sqrt{s}|\gamma'(s)|^2_{g(s)}ds\\
			&\ge -\frac{2k}{3}\tau^{\frac{3}{2}}+ c\int_0^\tau \sqrt{s}|\gamma'(s)|^2_{\bar g}ds
			\\
			&=-\frac{2k}{3}\tau^{\frac{3}{2}}+\frac{c}{2}\int_0^{\sqrt{\tau}}|\zeta'(\sigma)|^2_{\bar g}d\sigma,
		\end{align*}
		where we have applied the change of variable $\zeta(\sigma):=\gamma(\sigma^2)$. Then, by the Cauchy-Schwarz inequality, we have
		\begin{align*}
			2c^{-1}L&\ge -\frac{4k}{3c}\tau^{\frac{3}{2}}+\int_0^{\sqrt{\tau}}|\zeta'(\sigma)|^2_{\bar g}d\sigma
			\\
			&\ge -\frac{4k}{3c}\tau^{\frac{3}{2}}+ \frac{\left(\dist_{\bar g}(x_0,\gamma(s))\right)^2}{\sqrt{T}}
		\end{align*}
		for all $s\in[0,\tau]$. It follows that for any $L>0$, there exists $A=A(L,c,T)$, such that if $\gamma:[0,\tau]\to M$ satisfies $\mathcal{L}(\gamma)\le L$, then
		$$\gamma(s)\in B_{\bar g}(x_0,A)\quad\text{ for all }\quad s\in[0,\tau].$$
		In particular, $A$ is independent of the choice of $x_0$
		
		For any $x_0$ and $x$, let us fix an arbitrary piecewise $C^1$ curve $\beta:[0,\tau]\to M$  satisfying $\beta(0)=x_0$ and $\beta(\tau)=x$. Let $L=\mathcal L(\beta)$. Then we have that every minimizing sequence $\{\gamma_i\}_{i=1}^\infty$ of the $\mathcal L$-energy with $\mathcal L(\gamma_i)\le L$ from $(x_0,0)$ to $(x,\tau)$ is contained in $\overline B_{g(0)}(x_0,A)\times [0,\tau]$, which is a compact set in space-time. It follows that there exists at least one minimal $\mathcal{L}$-geodesic from $(x_0,0)$ to $(x,\tau)$.
		
		Finally, we show that the backward Ricci flow is $\mathcal{L}$-complete. Let $E\in M$ be an arbitrary compact set and fix a $\tau\in(0,T]$. By Lemma \ref{L-bounded in compact sets}, there is a constant $L(E,\tau)>0$, such that for any $x$, $y\in E$, there is a piecewise $C^1$-curve $\gamma:[0,\tau]\to M$ with $\gamma(0)=x$, $\gamma(\tau)=y$, and $\mathcal L(\gamma)$ not exceeding $L(E,\tau)$. The argument above shows that there exists at least one minimal $\mathcal L$-geodesic connecting $(x,0)$ and $(y,\tau)$, and all such minimal $\mathcal L$-geodesics are contained in a compact set.
	\end{proof}
	
	\subsubsection{Logarithmic lower bound for the Ricci curvature}
	
	Let us move on to the proof of Theorem \ref{conditions_l-complete}(2). We assume that the backward Ricci flow in question has a logarithmic lower bound for the Ricci curvature. In particular, we let $k>0$, $\alpha>0$, $\beta>0$, $\gamma\in(0,1)$ be constants and $x_0$ be a fixed point, such that
	\begin{equation}\label{Rlowerbound1}
		R_{g(\tau)}(x)\geq -k\quad\text{for all}\quad (x,\tau)\in M\times[0,T]
	\end{equation}
	and
	\begin{align}\label{Rclogarithmicbound}
		\mathrm{Ric}_{g(\tau)}(x)\geq-\alpha (T-\tau)^{-\gamma}\cdot\ln\left(\dist_{g(\tau)}(x_0,x)+\beta\right)g(\tau)\quad\text{for all}\quad (x,\tau)\in M\times[0,T].
	\end{align}
	As in the previous subsection, we show the existence of a minimal $\mathcal L$-geodesic from $(x_0,0)$ to $(x,\tau)\in M\times(0,T]$, provided that $T$ is small enough.

	\begin{lem}[Distance distortion]\label{distdistortion}
		For any $\tau_1$ and $\tau_2\in[0,T]$ with $\tau_1\leq\tau_2$ and for any $x\in M$, we have
		\begin{align*}
			\dist_{g(\tau_1)}(x_0,x)+\beta\leq\left( \dist_{g(\tau_2)}(x_0,x)+\beta\right)^{\exp\left(\frac{\alpha}{1-\gamma} T^{1-\gamma}\right)}.
		\end{align*}
	\end{lem}
	\begin{proof}
		For any $\tau\in [0,T]$, let us define $r(\tau):=\dist_{g(\tau)}(x_0,x)$ and let $\gamma:[0,r(\tau)]\to M$ be a $g(\tau)$-geodesic connecting $x_0$ and $x$ with unit speed. We may apply (\ref{Rclogarithmicbound}) to compute
		\begin{align*}
			\frac{d r}{d\tau}&=\int_0^r\mathrm{Ric}_{g(\tau)}(\gamma'(s),\gamma'(s))ds
			\\
			&\ge -\alpha(T-\tau)^{-\gamma}\cdot r\ln(r+\beta)
			\\
			&\geq -\alpha(T-\tau)^{-\gamma}\cdot (r+\beta)\ln(r+\beta),
		\end{align*}
		and equivalently
		\begin{equation*}
			\frac{d}{d\tau}\ln\left(\ln\left(r+\beta\right)\right)\geq -\alpha(T-\tau)^{-\gamma}.
		\end{equation*}
		Integrating the above inequality from $\tau_1$ to $\tau_2$, we obtain the lemma.
	\end{proof}
	
	Next, we prove that a minimizing sequence of the $\mathcal L$-energy is always contained in a compact set in space-time.
	
	\begin{lem}\label{compactestimate2}
		There exists a positive number $\overline T=\overline T(\alpha,\gamma)>0$ with the following property. If $T<\overline T$, then for any $r>0$ and $L>0$, there is a positive number $A=A(r,L,T,\alpha,\beta,\gamma)$, such that the following holds. For any $(x,\tau)\in M\times(0,T]$ with $$\dist_{g(\tau)}(x_0,x)\leq r,$$ if $\gamma:[0,\tau]\to M$ is a piecewise $C^1$ curve satisfying $\gamma(0)=x_0$, $\gamma(\tau)=x$, and 
		\begin{align*}
			\mathcal{L}(\gamma)\leq L
		\end{align*}
		Then we have
		\begin{align*}
			\gamma(s)\in \overline B_{g(s)}(x_0,A)\quad\text{ for all }\quad s\in[0,\tau].
		\end{align*}
	\end{lem}
	
	\begin{proof}
		We shall argue by contradiction. Let $A\gg r$ be a large number to be fixed. Define
		\begin{align*}
			\bar\tau:=\inf\left\{\tau'\,\left|\,\gamma(s)\in B_{g(s)}(x_0,A)\text{ for all }s\in[\tau',\tau]\right\}\right.\in[0,\tau).
		\end{align*}
		We assume that $\bar\tau>0$ and show that there is a contradiction when $A$ is large enough. By this contradictory assumption, we have
		\begin{gather}\label{contradictory1}
			\dist_{g(\bar\tau)}(x_0,\gamma(\bar\tau))=A,
			\\\nonumber
			\dist_{g(s)}(x_0,\gamma(s))<A\quad\text{ for all }\quad s\in[\bar\tau,\tau].
		\end{gather}
		
		By Lemma \ref{distdistortion}, we have
		\begin{align*}
			\dist_{g(\tau_1)}(x_0,\gamma(\tau_2))\leq \left(\dist_{g(\tau_2)}(x_0,\gamma(\tau_2))+\beta\right)^{\exp\left(\frac{\alpha}{1-\gamma} T^{1-\gamma}\right)}\quad\text{ for all }\quad 0\leq \tau_1\leq\tau_2\leq \tau,
		\end{align*}
		and consequently
		\begin{gather}\label{contradictory2}
			\dist_{g(\tau_1)}(x_0,\gamma(\tau_2))\leq \left(A+\beta\right)^{\exp\left(\frac{\alpha}{1-\gamma} T^{1-\gamma}\right)}\quad\text{ for all }\quad \bar\tau \leq \tau_1\leq\tau_2\leq \tau
			\\
			\dist_{g(\bar\tau)}(x_0,x)\leq (r+\beta)^{\exp\left(\frac{\alpha}{1-\gamma} T^{1-\gamma}\right)}.\label{contradictory3}
		\end{gather}
		
		By (\ref{Rclogarithmicbound}) and (\ref{contradictory2}), we may estimate
		\begin{align*}
			\frac{d}{d{\tau_1}}|\gamma'(\tau_2)|^2_{g(\tau_1)}&=2\mathrm{Ric}_{g(\tau_1)}(\gamma'(\tau_2),\gamma'(\tau_2))
			\\
			&\geq -2\alpha(T-\tau_1)^{-\gamma}\cdot\ln\left((A+\beta)^{\exp\left(\frac{\alpha}{1-\gamma} T^{1-\gamma}\right)}+\beta\right)|\gamma'(\tau_2)|^2_{g(\tau_1)},
		\end{align*}
		for all $\bar\tau\le \tau_1\le\tau_2\le\tau$. Integrating over $\tau_1$, we have
		\begin{align}
			|\gamma'(\tau_2)|^2_{g(\tau_2)}&\geq\left((A+\beta)^{\exp\left(\frac{\alpha}{1-\gamma} T^{1-\gamma}\right)}+\beta\right)^{-\frac{2\alpha}{1-\gamma} T^{1-\gamma}}|\gamma'(\tau_2)|^2_{g(\bar\tau)}
			\\\nonumber
			&\geq (2A)^{-\frac{2\alpha}{1-\gamma} T^{1-\gamma}\cdot \exp\left(\frac{\alpha}{1-\gamma} T^{1-\gamma}\right)}|\gamma'(\tau_2)|^2_{g(\bar\tau)}\quad\text{ for all }\quad \bar\tau\leq \tau_2\le\tau,
		\end{align}
		if $A\geq\underline A(\alpha,\beta,\gamma, T)$ for some large positive constant $\underline A(\alpha,\beta,\gamma,T)$. By the definition of the $\mathcal L$-energy and the assumption (\ref{Rlowerbound}), we have
		\begin{align*}
			L&\ge \mathcal{L}(\gamma)\ge -\frac{2k}{3}\tau^{\frac{3}{2}}+ \int_{\bar\tau}^\tau\sqrt{s}|\gamma'(s)|^2_{g(s)}ds\geq  -\frac{2k}{3}\tau^{\frac{3}{2}}+(2A)^{-\frac{2\alpha}{1-\gamma} T^{1-\gamma}\cdot \exp\left(\frac{\alpha}{1-\gamma} T^{1-\gamma}\right)}\int_{\bar\tau}^\tau\sqrt{s}|\gamma'(s)|^2_{g(\bar\tau)}ds
		\end{align*}
		Letting $\zeta(\sigma):=\gamma(\sigma^2)$, we have
		\begin{align*}
			L&\geq  -\frac{2k}{3}T^{\frac{3}{2}}+(2A)^{-\frac{2\alpha}{1-\gamma} T^{1-\gamma}\cdot \exp\left(\frac{\alpha}{1-\gamma} T^{1-\gamma}\right)}\int_{\sqrt{\bar\tau}}^{\sqrt{\tau}}\tfrac{1}{2}|\zeta'(\sigma)|^2_{g(\bar\tau)}d\sigma
			\\
			&\ge -\frac{2k}{3}T^{\frac{3}{2}}+(2A)^{-\frac{2\alpha}{1-\gamma} T^{1-\gamma}\cdot \exp\left(\frac{\alpha}{1-\gamma} T^{1-\gamma}\right)} \cdot\frac{\left(\dist_{g(\bar\tau)}(x,\gamma(\bar{\tau}))\right)^2}{2\left(\sqrt{\tau}-\sqrt{\bar\tau}\right)}
			\\
			&\ge -\frac{2k}{3}T^{\frac{3}{2}}+(2A)^{-\frac{2\alpha}{1-\gamma} T^{1-\gamma}\cdot \exp\left(\frac{\alpha}{1-\gamma} T^{1-\gamma}\right)} \cdot\frac{\left(A-(r+\beta)^{\exp\left(\frac{\alpha}{1-\gamma} T^{1-\gamma}\right)}\right)^2}{2\sqrt{T}},
		\end{align*}
		where we have applied (\ref{contradictory1}) and (\ref{contradictory3}). Finally, if we take $T<\overline T(\alpha,\gamma)$ such that $\frac{2\alpha}{1-\gamma} T^{1-\gamma}\cdot \exp\left(\frac{\alpha}{1-\gamma} T^{1-\gamma}\right)\leq 1$, then we have 
		\begin{align*}
			L\geq \frac{\left(A-(r+\beta)^{\exp\left(\frac{\alpha}{1-\gamma} T^{1-\gamma}\right)}\right)^2}{4A\sqrt{T}}-\frac{2k}{3}T^{\frac{3}{2}},
		\end{align*}
		and this obviously is a contradiction if $A>\underline A(r,L,T,\alpha,\beta,\gamma)$.
	\end{proof}

	\begin{prop}[Existence of minimal $\mathcal L$-geodesic]\label{existence}
		If $T<\overline T$, where $\overline T$ is the constant in the statement of Lemma \ref{compactestimate2}, then for any $(x,\tau)\in M\times(0,T]$, there is a minimal $\mathcal L$-geodesic from $(x_0,0)$ to $(x,\tau)$. In particular, the backward Ricci flow is $\mathcal{L}$-complete on $[0,T]$.
	\end{prop}
	
	\begin{proof}
		
		We fix an arbitrary $(x,\tau)\in M\times(0,T]$ and a piecewise $C^1$ curve $\gamma:[0,\tau]\to M$ with $\gamma(0)=x_0$ and $\gamma(\tau)=x$. Let
		\begin{align*}
			r&:=\dist_{g(\tau)}(x_0,x),
			\\
			L&:=\mathcal L(\gamma)
		\end{align*}
		and let $A=A(r,L,T,\alpha, \beta,\gamma)$ be the positive constant given by Lemma \ref{compactestimate2}. Then we have that, by Lemma \ref{compactestimate2}, any minimizing sequence $\{\gamma_i\}_{i=1}^\infty$ of the $\mathcal L$-energy with $\mathcal L(\gamma_i)\le L$ from $(x_0,0)$ to $(x,\tau)$ lies in the space-time compact set
		$$\bigcup_{s\in[0,\tau]}\overline B_{g(s)}(x_0,A)\times\{s\},$$
		and the existence of minimal $\mathcal L$-geodesic follows from the standard theory of variation.
		
		Finally, we need to show that $(M,g(\tau))_{\tau\in[0,T]}$ is $\mathcal{L}$-complete, where $T\le \overline T$ satisfies the condition in the previous Lemma. For $0\le \tau_1\le \tau_2\le T$, it is obvious that $(M,g(\tau+\tau_1))_{\tau\in[0,\tau_2-\tau_1]}$ still satisfies \eqref{Rlowerbound1} and \eqref{Rclogarithmicbound}. Thus we may take $\tau_1=0$ and $\tau_2=\bar\tau\in[0,T]$. Let $E$ be an arbitrary closed set in $M$. Since each time slice of the backward Ricci flow is complete, we may, without loss of generality, assume that $E=\overline{B}_{g(\bar\tau)}(x_0,r)$. Applying the standard distance distortion estimate, while using the fact that $|\mathrm{Ric}|$ is uniformly bounded in $\overline{B}_{g(\bar\tau)(x_0,r)}\times[0,\bar\tau]$, we have
		$$r_0:=\sup_{\tau\in[0,\bar\tau]}\sup_{y\in \overline{B}_{g(\bar\tau)(x_0,r)}}\dist_{g(\tau)}(x_0,y)<\infty.$$
		It clearly follows from the triangle inequality that, for any $y_0\in \overline{B}_{g(\bar\tau)(x_0,r)}$,
		\begin{align}\label{Rclogarithmicbound2}
			\mathrm{Ric}_{g(\tau)}(x)\geq-\alpha (T-\tau)^{-\gamma}\cdot\ln\left(\dist_{g(\tau)}(y_0,x)+\beta+r_0\right)g(\tau)\quad\text{for all}\quad (x,\tau)\in M\times[0,T].
		\end{align}
		Now we may argue as in the previous lemma with \eqref{Rclogarithmicbound} replaced by \eqref{Rclogarithmicbound2}, $x_0$ replaced by an arbitrary $x\in \overline{B}_{g(\bar\tau)}(x_0,r)$, and $\beta$ replaced by $\beta+r_0$ (note that $\overline{T}(\alpha,\gamma)$ is independent of $\beta$). There exists $A=A(2r,L(E,\bar\tau),T,\alpha,\beta+r_0,\gamma)>0$, where $L(E,\bar\tau)$ is obtained by applying Lemma \ref{L-bounded in compact sets} to $E=\overline{B}_{g(\bar\tau)}(x_0,r)$ and $\bar\tau\in[0,T]$, such that for $x$, $y\in \overline{B}_{g(\bar\tau)}(x_0,r)$, any minimal $\mathcal L$-geodesic connecting $(x,0)$ and $(y,\bar\tau)$ lies in the space-time compact set  
		$$\bigcup_{\tau\in[0,\bar\tau]}\overline B_{g(\tau)}(x,A)\times\{s\}\subset \bigcup_{\tau\in[0,\bar\tau]}\overline B_{g(\tau)}(x_0,A+r_0)\times\{s\},$$
		which is clearly contained in the cross product of a spatial compact set $F$ and the time interval $[0,\bar\tau]$.
	\end{proof}

	\section{Perelman's pseudolocality under the $\mathcal L$-complete assumption}

	In this section, we prove Theorem \ref{pseudo_locality}. Let $(M,g_t)_{t\in[0,T]}$ be the $\mathcal L$-complete Ricci flow in the statement of Theorem \ref{pseudo_locality} satisfying
	\begin{gather}
		R(g_0)\ge -1,\label{scalar lower}
		\\
		\left(\operatorname{Area}_{g_0}(\partial \Omega)\right)^{n} \geq(1-\delta) n^n\omega_{n}\left(\operatorname{Vol}_{g_0}(\Omega)\right)^{n-1}\  \text{ for all }\  \Omega\subset M.\label{isoperimetric}
	\end{gather}
	When formulating an argument of the reduced distance, we shall always consider the backward Ricci flow $g(\tau):=g_{T-\tau}$ instead. 
	
	Let $\ell: M\times[0,T]\to\mathbb{R}$ be the reduced distance based at some point on the $\tau=0$ (i.e. $t=T$) slice and $\mathcal V(\tau)$ the corresponding reduced volume. First of all, we shall prove some preparatory lemmas focusing on the ``final'' time-slice of the $\ell$-function 
	\begin{align}\label{l-final-slice}
		\ell:= \ell(\cdot,T):M^n\to \R.
	\end{align}
	
	\begin{lem}\label{grad}
		Under the $\mathcal L$-complete assumption and \eqref{scalar lower}, we have
		\begin{eqnarray}\label{l2-gradient}
			\int_{M^n}|\nabla \ell|^2(4\pi T)^{-\frac{n}{2}}e^{-\ell}d\mu\le \frac{2n}{T}+2+2T^{\frac{1}{2}}<\infty.
		\end{eqnarray}
		Consequently, $(4\pi T)^{-n/4}e^{-\ell/2}$ is a $W^{1,2}$ function on $M^n$.
	\end{lem}
	\begin{proof}
		This lemma is the same as \cite[Lemma 6.6]{CZ}. We shall include its proof for the convenience of the reader. Rewriting (\ref{eq_l_7}) at $\tau=T$, we have
		\begin{eqnarray}\label{eq-temp-1}
			\int_{M^n}\Big(-2\nabla \ell\cdot\nabla\phi-|\nabla \ell|^2\phi\Big)d\mu&\leq& -\int_{M^n}\left(R+\frac{\ell-n}{T}\right)\phi d\mu
			\\\nonumber
			&\leq&\left(\frac{n}{T}+1+T^{\frac{1}{2}}\right)\int_{M^n} \phi d\mu,
		\end{eqnarray}
		where we have used the facts that $R_{g(\tau)}\ge -1$ and that $\ell\geq -T^{\frac{3}{2}}$; to see that the latter is true, one needs only to recall the definition (\ref{l-energy})(\ref{reduced_l}) of the reduced distance and the fact that $R_{g(\tau)}\geq -1$ for all $\tau\in [0,T]$ .

		Next, we take
		\begin{eqnarray}\label{cutoff}
			\phi:=\varphi_{A}^2(4\pi T)^{-\frac{n}{2}}e^{-\ell},
		\end{eqnarray}
		where 
		$$\varphi_A(x)=\eta\left(\frac{d_{g(T)}(x,x_0)}{A}\right),$$
		$x_0$ is the fixed base point, and $\eta:[0,\infty)\to[0,1]$ is a smooth and decreasing function satisfying $\eta|_{[0,1]}\equiv 1$, $\eta|_{[2,\infty)}\equiv 0$, and $-2\le \eta'(s)\le 0$ for all $s\in[0,\infty)$. Obviously, $\phi$ is a valid test function of (\ref{eq_l_7}), since $\varphi_A$ is smooth and compactly supported and $\ell$ is locally Lipschitz. Then, (\ref{eq-temp-1}) becomes
		\begin{eqnarray*}
			\int_{M^n}\Big(\varphi_{A}^2|\nabla \ell|^2-4\langle\nabla \ell,\nabla\varphi_{A}\rangle\varphi_{A}\Big)(4\pi T)^{-\frac{n}{2}}e^{-\ell}d\mu\leq\left(\frac{n}{T}+1+T^{\frac{1}{2}}\right)\int_{M^n}\varphi_{A}^2(4\pi T)^{-\frac{n}{2}}e^{-\ell}d\mu.
		\end{eqnarray*}
		Hence we have
		\begin{eqnarray*}
			\int_{M^n}\varphi_{A}^2|\nabla \ell|^2(4\pi T)^{-\frac{n}{2}}e^{-\ell}d\mu&\leq&\left(\frac{n}{T}+1+T^{\frac{1}{2}}\right)\int_{M^n}\varphi_{A}^2(4\pi T)^{-\frac{n}{2}}e^{-\ell}d\mu
			\\
			&&+4\int_{M^n}\langle\nabla \ell,\nabla\varphi_{A}\rangle\varphi_{A}(4\pi T)^{-\frac{n}{2}}e^{-\ell}d\mu
			\\
			&\leq&\left(\frac{n}{T}+1+T^{\frac{1}{2}}\right)\int_{M^n}\varphi_{A}^2(4\pi T)^{-\frac{n}{2}}e^{-\ell}d\mu
			\\
			&&+\frac{1}{2}\int_{M^n}\varphi_{A}^2|\nabla \ell|^2(4\pi T)^{-\frac{n}{2}}e^{-\ell}d\mu
			\\
			&&+8\int_{M^n}|\nabla\varphi_{A}|^2_{g_T}(4\pi T)^{-\frac{n}{2}}e^{-\ell}d\mu,
		\end{eqnarray*}
		and subsequently
		\begin{eqnarray*}
			\int_M\varphi_{A}^2|\nabla \ell|^2(4\pi T)^{-\frac{n}{2}}e^{-\ell}d\mu&\leq& \left(\frac{64}{A^2}+\frac{2n}{T}+2+2T^{\frac{1}{2}}\right)\int_M(4\pi T)^{-\frac{n}{2}}e^{-\ell}d\mu
			\\
			&=&\left(\frac{64}{A^2}+\frac{2n}{T}+2+2T^{\frac{1}{2}}\right)\mathcal{V}(T)
			\\
			&\leq&\frac{64}{A^2}+\frac{2n}{T}+2+2T^{\frac{1}{2}},
		\end{eqnarray*}
		where we have applied the fact that $\mathcal V(T)\leq 1$ due to Theorem \ref{Monotonicity}(1). Taking $A\to \infty$, we obtain (\ref{l2-gradient}). The boundedness of the $L^2$-norm of $(4\pi T)^{-n/4}e^{-\ell/2}$ is but a consequence of the fact that $\mathcal V(T)\leq 1$. 
	\end{proof}

	\begin{lem}\label{reduced_volume_estimate}
		Under the $\mathcal L$-complete assumption and \eqref{scalar lower}\eqref{isoperimetric}, we have $$\mathcal V(T)\geq e^{-T}(1-\delta).$$ 
	\end{lem}
	
	\begin{proof}
		
		Rewriting (\ref{eq-temp-1}) using the same cut-off function as defined in 	(\ref{cutoff}) and the fact that $	R(g_0)\geq -1$, we have
		
		\begin{align*}
			&\int_{M^n}\left(|\nabla \ell|^2+\frac{\ell-n}{T}\right)\varphi_A^2(4\pi T)^{-\frac{n}{2}}e^{-\ell}d\mu
			\\
			\leq&\  4\int_{M^n}\langle\nabla\ell,\nabla\varphi_A\rangle\varphi_A(4\pi T)^{-\frac{n}{2}}e^{-\ell}d\mu
			-\int_{M}R\varphi_A^2(4\pi T)^{-\frac{n}{2}}e^{-\ell}d\mu
			\\
			\leq &\ 4\left(\int_{M^n}|\nabla \ell|^2(4\pi T)^{-\frac{n}{2}}e^{-\ell}d\mu\right)^{\frac{1}{2}}\left(\int_{M^n}|\nabla \varphi_A|^2(4\pi T)^{-\frac{n}{2}}e^{-\ell}d\mu\right)^{\frac{1}{2}}+\int_{M}\varphi_A^2(4\pi T)^{-\frac{n}{2}}e^{-\ell}d\mu
			\\
			\leq &\ 4\left(\frac{2n}{T}+2+2T^{\frac{1}{2}}\right)^{\frac{1}{2}}\left(\frac{4}{A^2}\mathcal V(T)\right)^{\frac{1}{2}}
			+\int_{M}\varphi_A^2(4\pi T)^{-\frac{n}{2}}e^{-\ell}d\mu,
		\end{align*}
		where we have applied Lemma \ref{grad}. Taking $A\to\infty$, we have
		\begin{equation}\label{reversesobolev}
			\int_{M^n}\left(T|\nabla\ell|^2+\ell-n\right)(4\pi T)^{-\frac{n}{2}}e^{-\ell}d\mu\leq T\int_{M}(4\pi T)^{-\frac{n}{2}}e^{-\ell}d\mu.
		\end{equation}

		Rescaling $g(\tau)$ as 
		$$\tilde{g}(\tau)=(2T)^{-1}g(2T\tau).$$
		and denote by $\tilde{\ell}$ the reduced length with respect to $\tilde{g}(\tau)$. 	
		Then  (\ref{reversesobolev}) becomes
		\begin{equation}\label{reversesobolev2}
			\int_{M^n}\left(\tfrac{1}{2}|\nabla \tilde{\ell}_{1/2}|^2+\tilde{\ell}_{1/2}-n\right)(2\pi )^{-\frac{n}{2}}e^{-\tilde{\ell}_{1/2}}d\tilde{\mu}\leq T\int_{M}(2\pi )^{-\frac{n}{2}}e^{-\tilde{\ell}_{1/2}}d\tilde{\mu}.
		\end{equation}	
		Letting
		$$\varphi(x):=(2\pi)^{-\frac{n}{4}}e^{-\frac{1}{2} \tilde{\ell}_{1/2}(x)}\in W^{1,2}(M),$$
		we can rewrite (\ref{reversesobolev2}) as
		\begin{align*}
			& \int_{M}\left(2|\nabla\varphi|^2-\varphi^2\log\varphi^2\right)d\tilde{\mu}\\
			\leq&\ \left(n+\frac{n}{2}\log(2\pi)\right)\int_{M}\varphi^2d\tilde{\mu}+T\int_{M}\varphi^2d\tilde{\mu}
			\\
			\leq &\ \int_{M}\left(2|\nabla\varphi|^2-\varphi^2\log\varphi^2\right)d\tilde{\mu}+\left(\log\left(\int_{M}\varphi^2 d\tilde{\mu}\right)-\log(1-\delta)+T\right)\int_{M}\varphi^2 d\tilde{\mu},
		\end{align*}
		where in the second inequality we have applied Theorem \ref{logsob}. Consequently, we have
		$$\int_{M}\varphi^2 d\tilde{\mu}\geq e^{-T}(1-\delta).$$
		
	\end{proof}

	Now we proceed to prove Theorem \ref{pseudo_locality} by contradiction. Assume that there exist $\alpha $ and sequences of positive numbers $\delta_{i} \rightarrow 0^{+}, \varepsilon_{i} \rightarrow 0^{+}$, and smooth $\mathcal L$-complete pointed Ricci flows $\left(\mathcal{M}_{i}^{n}, g_{i,t}\right)_{t \in\left[0, \varepsilon_{i}^{2}\right]}, i\in\mathbb{N},$ satisfying \eqref{scalar lower} \eqref{isoperimetric}.
	Furthermore, there exist ``bad'' points in space-time $(x_{i},t_i) \in \mathcal{M}_{i}\times \left(0, \varepsilon_{i}^{2}\right]$ satisfying
	\begin{align}\label{counterstatement}
		\left|\mathrm{Rm}_{g_{i}}\right|\left(x_{i}, t_{i}\right)>\frac{\alpha}{t_{i}}
	\end{align}
	for all $i \in \mathbb{N}$.

	We can apply a point-picking method as in Perelman's proof of pseudolocality theorem. The details can be found in \cite{RFV3}. Arguing in the same way as \cite[Lemma 21.12]{RFV3} and \cite[Chapter 22 $\S 1$]{RFV3}, we may adjust $\left(x_{i}, t_i\right)$ such that they satisfy not only \eqref{counterstatement}, but also
	\begin{align}\label{parabolic curvature bound}
		\left|\mathrm{Rm}_{g_{i}}\right|(x, t) \leq 4 Q_{i}\qquad\text{ for all }\qquad (x,t)\in B_{g_{i,t_i}}\left(x_i,A_i Q_{i}^{-1 / 2}\right)\times \left[t_i-\tfrac{\alpha}{2} Q_{i}^{-1},t_i\right]
	\end{align}
	where $Q_{i}:= \left|\mathrm{Rm}_{g_{i}}\right|(x_i, t_i)>\frac{\alpha}{t_i}$ and $A_i\to \infty$.
	
	\begin{lem}\label{reduced_volume}
		For each $i$, there exists a time
		$$
		\tilde{t}_{i} \in\left[t_i-\frac{\alpha}{2} Q_{i}^{-1}, t_i\right)
		$$
		such that
		$$
		\mathcal {V}_i(t_i-\tilde{t}_{i}) \leq 1-\beta_{1}
		$$
		where $\mathcal {V}_i$ is the reduced volume of $g_i$ based at  $\left(x_{i}, t_i\right)$ and $\beta_{1}\in(0,1)$ is a constant independent of $i$.
	\end{lem}
	
	\begin{proof}
		We apply the following rescaling and time-shifting to the Ricci flows
		\begin{equation}\label{rescaling}
			g_{i,t} \longrightarrow \hat{g}_{i,t} \doteqdot Q_{i} g_{i,t_i+Q_{i}^{-1} t},
		\end{equation}
		and proceed to show that for each $i\in\mathbb{N}$, there is a $\hat{t}_i\in [-\frac{\alpha}{2},0)$ with
		$$
		\mathcal {\hat{V}}_i(-\hat{t}_i)   \leq 1-\beta_{1},
		$$
		where $\mathcal{\hat V}_i$ is the reduced volume of $\hat{g}_i$ based at $(x_i,0)$. \eqref{parabolic curvature bound} now becomes
		\begin{align}\label{parabolic curvature bound 1}
			\left|\mathrm{Rm}_{\hat g_{i}}\right|(x, t) \leq 4 \qquad\text{ for all }\qquad (x,t)\in B_{\hat g_{i,0}}\left(x_i,A_i\right)\times \left[-\tfrac{\alpha}{2},0\right].
		\end{align}
		
		We argue by contradiction. Assume that, by passing to a subsequence,
		\begin{align}\label{contradition V}
			\lim_{i\to\infty}\inf_{\tau\in(0,\frac{\alpha}{2}]}\mathcal {\hat{V}}_i(\tau)=1.
		\end{align}

		\noindent \textbf{Claim.} There are positive constants $c(\alpha)$ and $C(\alpha)$ with the following property. Fix any large $i\in\mathbb{N}$ and any $\bar\tau\in(0,\frac{\alpha}{4}]$. Let $\gamma_V:[0,\bar\tau]\to M$ be a minimal $\mathcal L$-geodesic with respect to $\hat g_i(\tau)=\hat g_{i,-\tau}$ starting from $(x_i,0)$ with $\lim_{\tau\to 0}\sqrt{\tau}\gamma'_V(\tau)=V\in T_{x_i}M$. For any $A\in [10C(\alpha),A_i]$, if $|V|_{\hat g_{i}(0)}\le c(\alpha)A$, then $\gamma_V\big|_{[0,\bar\tau]}\subset B_{\hat g_i(0)}(x_i,\tfrac{1}{2}A)$. 
		
		\begin{proof}[Proof of the Claim]
			Writing $X(\tau)=\gamma_V'(\tau)$, we may apply the $\mathcal L$-geodesic equation to obtain (c.f. \cite[(26.4)]{KL})
			\begin{align*}
				\frac{d}{d\tau}(\tau|X|^2_{\hat g_i(\tau)})=-2\tau\mathrm{\mathrm{Ric}}_{\hat g_i}(X,X)+\tau \langle X,\nabla R_{\hat g_i}\rangle_{\hat g_i(\tau)}.
			\end{align*}
			By \eqref{parabolic curvature bound 1} and Shi's estimates, we have 
			\begin{align*}
				\left|\mathrm{Rm}_{\hat g_i}\right|\le 4,\qquad \left|\nabla R_{\hat g_i}\right|\le C(n,\alpha),\qquad \text{for all}\qquad (x,\tau)\in B_{\hat g_i(0)}(x_i\tfrac{1}{2}A_i)\times[0,\tfrac{\alpha}{4}].
			\end{align*}
			Then, as long as $\gamma_V\big|_{[0,\tau]}\subset B_{\hat g_{i}(0)}(x_i,\tfrac{1}{2}A_i)$, we may estimate
			\begin{align*}
				\frac{d}{d\tau}\left(\tau|X|^2_{\hat g_i(\tau)}\right)\le 8\tau|X|^2_{\hat g_i(\tau)}+C\tau|X|_{\hat g_i(\tau)}\le 9\tau|X|^2_{\hat g_i(\tau)}+C
			\end{align*}
			and by integrating
			\begin{align}\label{bound11}
				\tau|X|^2_{\hat g_i(\tau)}\le e^{4\alpha}|V|^2_{\hat g_i(0)}+C.
			\end{align}
			
			Let us fix an $A\le A_i$. Assume that $\tau'\in (0,\tfrac{\alpha}{4})$ is the first time when $\gamma_V$ has reached the boundary of $B_{\hat g_i(0)}(x_i,\frac{1}{2}A)$. Thus, $\gamma_V\big|_{[0,\tau']}\subset B_{\hat g_i(0)}(x_i,\frac{1}{2}A)$, $\dist_{\hat g_i(0)}(\gamma_V(0),\gamma_V(\tau'))=\frac{1}{2}A$, and 
			\begin{align*}
				|X(\tau)|_{\hat g_i(\tau)}^2\ge e^{-\alpha}|X(\tau)|^2_{\hat g_i(0)}\qquad\text{ for all }\qquad \tau\in[0,\tau'].
			\end{align*}
			Then,
			\begin{align*}
				\tfrac{1}{2}A&=\dist_{\hat g_i(0)}(\gamma_V(0),\gamma_V(\tau')) \le \int_0^{\tau'}|X(\tau)|_{\hat g_i(0)}d\tau\le e^{\alpha/2}\int_0^{\tau'}|X(\tau)|_{\hat g_i(\tau)}d\tau
				\\
				&\le e^{\alpha/2}\left(\int_0^{\tau'}\sqrt{\tau}|X(\tau)|^2_{\hat g_i(\tau)}d\tau\right)^{\frac{1}{2}}\left(\int_0^{\tau'}\tau^{-\frac{1}{2}}d\tau\right)^{\frac{1}{2}}
				\\
				&\le C(\alpha)\left(\int_0^{\tau'}\tau^{-\frac{1}{2}}\left(e^{4\alpha}|V|^2_{\hat g_i(0)}+C\right)d\tau\right)^{\frac{1}{2}}
				\\
				&\le C(\alpha)(|V|_{\hat g_i(0)}+1),
			\end{align*}
			where we have applied \eqref{bound11}. If $A\ge 10 C(\alpha)$, then the above computation yields a contradiction if we assume $|V|_{\hat g_i(0)}\le \frac{1}{4C(\alpha)}A$. This proves the claim.
		\end{proof}
		
		By contradictory assumption \eqref{contradition V} and the proof of Perelman's no local collapsing theorem (c.f. \cite[Theorem 26.2]{KL}), it is easy to see that
		\begin{align*}
			\inf_{i}\operatorname{inj}_{g_{i,0}}(x_i)>0.
		\end{align*}
		Therefore, we may extract a subsequence from $\big\{\big(M_i,\hat g_{i,t},x_i\big)_{t\in[-\frac{\alpha}{2},0]}\big\}_{i=1}^\infty$ which converges to a smooth Ricci flow with bounded curvature
		\begin{align*}
			(M_\infty,\hat g_{\infty,t},x_\infty)_{t\in(-\frac{\alpha}{2},0]}.
		\end{align*}
		Since $|\mathrm{Rm}_{\hat g_\infty}|(x_\infty,0)=\lim_{i\to\infty}|\mathrm{Rm}_{\hat g_i}|(x_i,0)=\lim_{i\to\infty}Q_i^{-1}|\mathrm{Rm}_{ g_i}|(x_i,t_i)=1$, $\hat g_\infty$ is not flat. 
		
		On the other hand, letting $\mathcal Lexp_{x_i}^\tau$ be the $\mathcal L$-exponential map of $\hat g_i$ based at $(x_i,0)$. The claim above shows that
		\begin{align*}
			M\setminus B_{\hat g_i(0)}(x_i,\tfrac{1}{2}A)\subset \mathcal Lexp_{x_i}^\tau\Big(\mathbb{R}^n\setminus\{|V|\le c A\}\Big),\ \ \ \text{ for all }\ \tau\in [0,\tfrac{\alpha}{4}]\ \text{ and }\ A\in [10C(\alpha),A_i].
		\end{align*}
		By the monotonicity of the $\mathcal L$-Jacobian \eqref{eq_l_8} and the fact that
		$$\lim_{\tau\to0}(4\pi \tau)^{-\frac{n}{2}}
		e^{-l(\gamma_V(\tau),\tau)}
		\mathcal{L}J_V(\tau)=(4\pi)^{-\frac{n}{2}}\exp\left(-\tfrac{|V|^2}{4}\right),$$
		we have that, for each $\tau\in(0,\frac{\alpha}{4}]$,
		\begin{align*}
			\int_{M\setminus B_{\hat g_i(0)}(x_i,\frac{1}{2}A)}(4\pi\tau)e^{-\ell_{\hat g_i}}(\cdot,\tau)d\hat g_i(\tau)\le \int_{|V|\ge cA}(4\pi)^{-\frac{n}{2}}\exp\left(-\tfrac{|V|^2}{4}\right)dV,\ \ \ \text{ for all }\  A\in [10C(\alpha),A_i].
		\end{align*}
		It is then clear from \eqref{contradition V} that
		$$\mathcal V_\infty(\tau)=\lim_{i\to\infty} \mathcal V_i(\tau)=1\qquad \text{ for all }\qquad \tau\in(0,\tfrac{\alpha }{4}],$$
		where $\mathcal V_\infty$ is the reduced volume of $\hat g_\infty$ based at $(x_\infty,0)$. Thus, $\hat g_\infty$ is the flat Euclidean space; this is a contradiction.
	\end{proof}
	
	Now we can give the proof of Theorem \ref{pseudo_locality}.
	\begin{proof}[Proof of Theorem \ref{pseudo_locality}] Applying Lemma \ref{reduced_volume_estimate} to the reduced volumes $\mathcal V_i$ of the Ricci flows $(M_i,g_i)$ based at $(x_i,t_i)$, since $t_i\le \varepsilon_i^2$, we have that 
		$$\mathcal V_i(t_i)\ge e^{-\varepsilon_i^2}(1-\delta_i).$$
		By the monotonicity of the reduced volume, we also have
		$$\mathcal V_i(\tau)\ge e^{-\varepsilon_i^2}(1-\delta_i)\qquad\text{for all }\qquad \tau\in(0,t_i].$$
		This is clearly a contradiction to Lemma \ref{reduced_volume} when $i$ is large.
		
	\end{proof}
	
	Theorem \ref{uniqueness} just follows from Theorem \ref{pseudo_locality} and Theorem 3.1 in \cite{C}.

	\section{uniqueness on Eulidean space}

	\begin{proof}[Proof of Theorem \ref{main0}]
		
		By Theorem \ref{conditions_l-complete}, the Ricci flow in question is $\mathcal L$-complete (in the second case, we also assume $T\le \overline{T}(\alpha,\gamma)$ for the moment). Then Theorem \ref{uniqueness} implies that the Ricci flow is one with bounded curvature. The conclusion then follows from \cite{CL}.
		
		We remark here that in case (2), once we established the strong uniqueness assuming $T\le \overline{T}(\alpha,\gamma)$, then the same result holds also for any $T>0$ by splitting $[0,T]$ into smaller intervals, each with length no greater than $\overline{T}(\alpha,\gamma)$.
	\end{proof}

\end{document}